\documentclass{amsart}

\usepackage{amssymb}
\usepackage{amsthm}
\usepackage{amsmath}
\usepackage{mathtools}
\usepackage{enumitem}
\usepackage{stmaryrd}
\usepackage{hyperref}
\theoremstyle{plain}
 \newtheorem{theorem}{Theorem}%[section]
 \newtheorem{lemma}[theorem]{Lemma}
 \newtheorem{corollary}[theorem]{Corollary}
 \newtheorem{claim}[theorem]{Claim}
\theoremstyle{definition}
 \newtheorem{definition}[theorem]{Definition}

 \newtheorem{remark}[theorem]{Remark}
 
\numberwithin{equation}{section}

\renewcommand{\a}{\alpha}

\renewcommand{\k}{\kappa}
\renewcommand{\o}{\omega}

\newcommand{\ps}{\emptyset}
\newcommand{\sledi}{\Rightarrow}
\newcommand{\rest}{\upharpoonright}

\renewcommand{\c}{\mathfrak c}
\newcommand{\cu}{\mathcal U}
\newcommand{\cv}{\mathcal V}

\newcommand{\seq}[1]{\left<#1\right>}

\newcommand{\set}[1]{\left\{#1\right\}}
\newcommand{\abs}[1]{\left\vert#1\right\vert}
\newcommand{\br}[1]{\left(#1\right)}

\newcommand{\range}{\operatorname{range}}

\renewcommand{\subset}{\subseteq}

\author[B. Kuzeljevic]{Borisa Kuzeljevic}
\address[Kuzeljevic]{Department of Mathematics and Informatics, Faculty of Sciences, University of Novi Sad, Serbia.}
\thanks{The first author was partially supported by the Science Fund of the Republic of Serbia grant no. 6062228}
\curraddr{}
\email{\href{borisha@dmi.uns.ac.rs}{borisha@dmi.uns.ac.rs}}
\urladdr{\url{https://people.dmi.uns.ac.rs/~borisha}}
\dedicatory{}

\author[D. Raghavan]{Dilip Raghavan}
\thanks{This paper was completed when the second author was a Fields Research Fellow.
The second author thanks the Fields Institute for its kind hospitality.}
\address[Raghavan]{Department of Mathematics \\
 National University of Singapore\\
 Singapore 119076.}
\email{\href{dilip.raghavan@protonmail.com}{dilip.raghavan@protonmail.com}}
\urladdr{\url{https://dilip-raghavan.github.io/}}

\author[J. L. Verner]{Jonathan L. Verner}
\address[Verner]{Department of Logic, Faculty of Arts, Charles University, Prague, Czech Republic.}
\curraddr{}
\email{\href{jonathan.verner@matfyz.cz}{jonathan.verner@matfyz.cz}}
\urladdr{\url{https://jonathan.temno.eu/}}
\dedicatory{}

\title[Lower bounds of sets of P-points]{Lower bounds of sets of P-points}
\keywords{ultrafilter, Rudin-Keisler order, P-point}
\subjclass[2010]{03E50, 03E05, 03E20, 03E35, 03E40, 54D35, 54D80}
\thanks{}%Both authors were partially supported by National University of Singapore research grant number R-146-000-211-112. The first author was partially supported by the MPNTR grant ON174006}
\date{\today}

\begin{document}

\begin{abstract}
We show that MA$_{\k}$ implies that each collection of ${P}_{\c}$-points of size at most $\k$ which has a $P_{\c}$-point as an $RK$ upper bound also has a ${P}_{\c}$-point as an $RK$ lower bound.
\end{abstract}

\maketitle

\section{Introduction}

The Rudin-Keisler ($RK$) ordering of ultrafilters has received considerable attention since its introduction in the 1960s. For example, one can take a look at \cite{rudin,Rudin:66,mer,blass,kn,RS,RK}, or \cite{RV}.
%The first systematic analysis of this ordering was done by Blass in his PhD thesis \cite{blassphd}.
Recall the definition of the Rudin-Keisler ordering.

\begin{definition}
	Let $\cu$ and $\cv$ be ultrafilters on $\o$. We say that $\cu\le_{RK}\cv$ if there is a function $f$ in $\o^{\o}$ such that $A\in \cu$ if and only if $f^{-1}(A)\in\cv$ for every $A\subset \o$.
\end{definition}

When $\cu$ and $\cv$ are ultrafilters on $\o$ and $\cu\le_{RK}\cv$, we say that $\cu$ is \emph{Rudin-Keisler (RK) reducible} to $\cv$, or that $\cu$ is \emph{Rudin-Keisler (RK) below} $\cv$.
In case $\cu\le_{RK}\cv$ and $\cv\le_{RK}\cu$ both hold, then we say that $\cu$ and $\cv$ are \emph{Rudin-Keisler equivalent}, and write $\cu\equiv_{RK}\cv$.

Very early in the investigation of this ordering of ultrafilters, it was noticed that the class of P-points is particularly interesting.
Recall that an ultrafilter $\cu$ on $\o$ is called a \emph{P-point} if for any $\set{a_n:n<\o}\subset\cu$ there is an $a\in\cu$ such that $a\subset^* a_n$ for every $n<\o$, i.e. the set $a\setminus a_n$ is finite for every $n<\o$.
P-points were first constructed by Rudin in \cite{rudin}, under the assumption of the Continuum Hypothesis.
The class of P-points forms a downwards closed initial segment of the class of all ultrafilters.
In other words, if $\cu$ is a P-point and $\cv$ is any ultrafilter on $\o$ with $\cv \; {\leq}_{RK} \; \cu$, then $\cv$ is also a P-point.
Hence understanding the order-theoretic structure of the class of P-points can provide information about the order-theoretic structure of the class of all ultrafilters on $\o$.
One of the first systematic explorations of the order-theoretic properties of the class of all ultrafilters, and particularly of the class of P-points, under ${\leq}_{RK}$ was made by Blass in \cite{blassphd} and \cite{blass}, where he proved many results about this ordering under the assumption of Martin's Axiom (MA).

Let us note here that it is not possible to construct P-points in ZFC only, as was proved by Shelah (see \cite{proper}).
Thus some set-theoretic assumption is needed to ensure the existence of P-points.
The most commonly used assumption when studying the order-theoretic properties of the class of P-points is MA.
Under MA every ultrafilter has character $\c$.
Therefore, the ${P}_{\c}$-points are the most natural class of P-points to focus on under MA.
Again, the ${P}_{\c}$-points form a downwards closed subclass of the P-points.
\begin{definition} \label{def:pc}
 An ultrafilter $\cu$ on $\o$ is called a $P_{\mathfrak c}$\emph{-point} if for every $\a<\mathfrak c$ and any $\set{a_i:i<\a}\subset \cu$ there is an $a\in\cu$ such that $a\subset^* a_i$ for every $i<\a$.
\end{definition}
In Theorem 5 from \cite{blass}, Blass proved in ZFC that if $\set{\cu_n:n<\o}$ is a countable collection of P-points and if there is a P-point $\cv$ such that $\cu_n\le_{RK}\cv$ for every $n<\o$, then there is a P-point $\cu$ such that $\cu\le_{RK}\cu_n$ for every $n<\o$.
In other words, if a countable family of P-points has an upper bound, then it also has a lower bound.

The main result of this paper generalizes Blass' theorem to families of ${P}_{\c}$-points of size less than $\c$ under MA.
More precisely, if MA holds and a family of ${P}_{\c}$-points of size less than $\c$ has an $RK$ upper bound which is a ${P}_{\c}$-point, then the family also has an RK lower bound.

Blass proved his result via some facts from \cite{blassmodeltheory} about non-standard models of complete arithmetic.
In order to state these results, we introduce a few notions from \cite{blassmodeltheory}.
The language $L$ will consist of symbols for all relations and all functions on $\o$.
Let $N$ be the standard model for this language, its domain is $\o$ and each relation or function denotes itself.
Let $M$ be an elementary extension of $N$, and let ${}^*R$ be the relation in $M$ denoted by $R$, and let ${}^*f$ be the function in $M$ denoted by $f$.
Note that if $a\in M$, then the set $\set{{}^*f(a):f:\o\to\o}$ is the domain of an elementary submodel of $M$.
Submodel like this, i.e.\@ generated by a single element, will be called \emph{principal}.
It is not difficult to prove that a principal submodel generated by $a$ is isomorphic to the ultrapower of the standard model by the ultrafilter ${\cu}_{a} = \{X \subset \omega: a \in {}^{\ast}{X}\}$.
If $A,B\subset M$, we say that they are \emph{cofinal with each other} iff $(\forall a\in A)(\exists b\in B)\ a\ {}^*\!\!\le b$ and $(\forall b\in B)(\exists a\in A)\ b\ {}^*\!\!\le a$.
Finally, we can state Blass' theorem.

\begin{theorem}[Blass, Theorem 3 in \cite{blassmodeltheory}]\label{t:blassmodel}
	Let $M_i$ ($i<\o$) be countably many pairwise cofinal submodels of $M$.
	Assume that at least one of the $M_i$ is principal.
	Then $\bigcap_{i<\o}M_i$ is cofinal with each $M_i$, in fact it contains a principal submodel cofinal with each $M_i$.
\end{theorem}

After proving this theorem, Blass states that it is not known to him whether Theorem \ref{t:blassmodel} can be extended to larger collections of submodels.
The proof of our main result clarifies this, namely in Theorem \ref{theorem3} below we prove that under MA it is possible to extend it to collections of models of size less than $\c$ provided that there is a principal model that is isomorphic to an ultrapower induced by a ${P}_{\c}$-point.
Then we proceed and use this result to prove Theorem \ref{maintheorem} where we extend Theorem 5 from \cite{blass} to collections of fewer than $\c$ many ${P}_{\c}$-points.

%Thus fix once and for all an uncountable regular cardinal $\k<\mathfrak c$.
Recall that MA$_{\a}$ is the statement that for every partial order $P$ which satisfies the countable chain condition and for every collection $\mathcal{D} = \{{D}_{i}: i < \a\}$ of dense subsets of $P$, there is a filter $G\subset P$ such that $G\cap {D}_{i}\neq\ps$ for every $i < \a$.

\section{The lower bound}

In this section we prove the results of the paper.
We begin with a purely combinatorial lemma about functions.

\begin{definition}\label{closedness}
 Let $\alpha$ be an ordinal, let $\mathcal F=\set{f_i:i<\a}\subset \o^{\o}$ be a family of functions, and let $A$ be a subset of $\a$.
 We say that a set $F\subset \o$ is ($A,\mathcal F$)\emph{-closed} if $f_i^{-1}(f_i''F)\subset F$ for each $i\in A$.
\end{definition}

\begin{remark}
 Notice that if $F$ is ($A,\mathcal F$)-closed, then $f_i^{-1}(f_i''F)=F$ for each $i\in A$.
\end{remark}

\begin{lemma}\label{finiteunion}
 Let $\a$ be an ordinal, let $\mathcal F=\set{f_i:i<\a}\subset \o^{\o}$ be a family of functions, and let $A$ be a subset of $\a$.
 Suppose that $m<\o$, and that $F_k$ is ($A,\mathcal F$)-closed subset of $\o$, for each $k<m$.
 Then the set $F=\bigcup_{k<m}F_k$ is ($A,\mathcal F$)-closed.
\end{lemma}
 
\begin{proof}
 To prove that $F$ is ($A,\mathcal F$)-closed take any $i\in A$, and $n\in f_i^{-1}(f_i''F)$.
 This means that there is some $n'\in F$ such that $f_i(n)=f_i(n')$.
 Let $k<m$ be such that $n'\in F_k$.
 Then $n\in f_i^{-1}(f_i''F_k)$.
 Since $F_k$ is ($A,\mathcal F$)-closed, $n\in f_i^{-1}(f_i''F_k)\subset F_k$.
 Thus $n\in F_k\subset F$.
\end{proof}

\begin{lemma}\label{ensuringcondition}
 Let $\a<\mathfrak c$ be an ordinal.
 Let $\mathcal F=\set{f_i:i<\a}\subset \o^{\o}$ be a family of finite-to-one functions.
 Suppose that for each $i,j<\a$ with $i<j$, there is $l<\o$ such that $f_j(n)=f_j(m)$ whenever $f_{i}(n)=f_{i}(m)$ and $n,m\ge l$.
 Then for each finite $A\subset \a$, and each $n<\o$, there is a finite $(A,\mathcal F)$-closed set $F$ such that $n\in F$. 
\end{lemma}

\begin{proof}
 First, if $A$ is empty, then we can take $F=\set{n}$.
 So fix a non-empty finite $A\subset\a$, and $n<\o$.
 For each $i,j\in A$ such that $i<j$, by the assumption of the lemma, take $l_{ij}<\o$ such that for each $n,m\ge l_{ij}$, if $f_i(n)=f_i(m)$, then $f_j(n)=f_j(m)$.
 Since $A$ is a finite set, there is $l=\max\set{l_{ij}:i,j\in A,\ i<j}$.
 So $l$ has the property that for every $i,j\in A$ with $i<j$, if $f_i(n)=f_i(m)$ and $n,m\ge l$, then $f_j(n)=f_j(m)$.
 
 Let $i_0=\max(A)$.
 Clearly, $f_i''l$ is finite for each $i\in A$, and since each $f_i$ is finite-to-one the set ${f}^{-1}_{i}\br{f_i''l}$ is finite for every $i\in A$.
 Since the set $A$ is also finite, there is $l'<\o$ such that $\bigcup_{i\in A}f_i^{-1}\br{f_i''l}\subset l'$.
 Again, since $f_{i_0}$ is finite-to-one there is $l''<\o$ such that $f^{-1}_{i_0}\br{f_{i_0}''l'}\subset l''$.
 Note that by the definition of numbers $l'$ and $l''$, we have $l''\ge l'\ge l$.
 
 \begin{claim}\label{closedset}
  For all $k<\o$, if $k\ge l''$, then the set $f_{i_0}^{-1}(f_{i_0}''\set{k})$ is ($A,\mathcal F$)-closed.
 \end{claim}
 
 \begin{proof}
  Fix $k\ge l''$ and let $X=f_{i_0}^{-1}(f_{i_0}''\set{k})$.
  First observe that $X\cap l'=\ps$.
  To see this suppose that there is $m\in X\cap l'$.
  Since $m\in X$, $f_{i_0}(m)=f_{i_0}(k)$. 
  Together with $m\in l'$, this implies that $k\in f_{i_0}^{-1}(f_{i_0}''\set{m})\subset f_{i_0}^{-1}(f_{i_0}''l')\subset l''$.
  Thus $k<l''$ contradicting the choice of $k$.
  Secondly, observe that if $m<l$ and $k'\in X$, then $f_i(m)\neq f_i(k')$ for each $i\in A$.
  To see this, fix $m<l$ and $k'\in X$, and suppose that for some $i\in A$, $f_i(m)=f_i(k')$.
  This means that $k'\in f_i^{-1}(f_i''\set{m})\subset f_i^{-1}(f_i''l)\subset l'$ contradicting the fact that $X\cap l'=\ps$.
  
  Now we will prove that $X$ is ($A,\mathcal F$)-closed.
  Take any $i\in A$ and any $m\in f_i^{-1}(f_i''X)$.
  We should prove that $m\in X$.
  Since $m\in f_i^{-1}(f_i''X)$, $f_i(m)\in f_i''X$ so there is some $k'\in X$ such that $f_i(m)=f_i(k')$.
  By the second observation, $m\ge l$.
  By the first observation $k'\ge l'\ge l$.
  By the assumption of the lemma, since $m,k'\ge l$, and $f_i(m)=f_i(k')$, it must be that $f_{i_0}(m)=f_{i_0}(k')$.
  Since $k'\in X=f_{i_0}^{-1}(f_{i_0}''\set{k})$, it must be that $f_{i_0}(k)=f_{i_0}(k')=f_{i_0}(m)$.
  This means that $m\in f_{i_0}^{-1}(f_{i_0}''\set{k})=X$ as required.
  Thus $f_{i_0}^{-1}(f_{i_0}''\set{k})$ is $(A,\mathcal F)$-closed.
 \end{proof}

 Now we inductively build a tree $T\subset \o^{<\o}$ we will be using in the rest of the proof.
 Fix a function $\Phi:\o\to\o^{<\o}$ so that $\Phi^{-1}(\sigma)$ is infinite for each $\sigma\in \o^{<\o}$.
 For each $m<\o$ let $u_m=\Phi(m)\br{\abs{\Phi(m)}-1}$, i.e. $u_m$ is the last element of the sequence $\Phi(m)$.
 Let $T_0=\set{\ps,\seq{n}}$ (recall that $n$ is given in the statement of the lemma).
 Suppose that $m\ge 1$, and that $T_m$ is given.
 If $\Phi(m)$ is a leaf node of $T_m$, then let
 \[\textstyle
  Z_m=\left(\bigcup_{i\in A}f_i^{-1}(f_i''\set{u_m})\right)\setminus\left(\bigcup_{\eta\in T_m}\range(\eta)\right),
 \]
 and $T_{m+1}=T_m\cup\set{\Phi(m)^{\frown}\seq{k}:k\in Z_m}$.
 If $\Phi(m)$ is not a leaf node of $T_m$, then $T_{m+1}=T_m$.
 Finally, let $T=\bigcup_{m<\o}T_m$ and $F=\bigcup_{\eta\in T}\range(\eta)$.
 
 \begin{claim}\label{subclaimtree}
  If $\sigma$ is a non-empty element of the tree $T$, then there is $m_0\ge 1$ such that $\sigma$ is a leaf node of $T_{m_0}$, that $\sigma=\Phi(m_0)$ and that $$\textstyle\bigcup_{i\in A}f_i^{-1}(f_i''\set{u_{m_0}})\subset \bigcup_{\eta\in T_{m_0+1}}\range(\eta).$$  
 \end{claim}
 
 \begin{proof}
  Fix a non-empty $\sigma$ in $T$.
  Let $m_1=\min\set{k<\o:\sigma\in T_k}$.
  Since $\abs{\sigma}>0$, $\sigma$ is a leaf node of $T_{m_1}$.
  Consider the set $W=\set{m\ge m_1:\Phi(m)=\sigma}$.
  Since the set $\set{m<\o:\Phi(m)=\sigma}$ is infinite, $W$ is non-empty subset of positive integers, so it has a minimum.
  Let $m_0=\min W$.
  Note that if $m_0=m_1$, then $\sigma$ is a leaf node of $T_{m_0}$.
  If $m_0>m_1$, by the construction of the tree $T$, since $\Phi(k)\neq \sigma$ whenever $m_1\le k<m_0$, it must be that $\sigma$ is a leaf node of every $T_k$ for $m_1<k\le m_0$.
  Thus $\sigma$ is a leaf node of $T_{m_0}$ and $\Phi(m_0)=\sigma$.
  Again by the construction of the tree $T$, we have $T_{m_0+1}=T_{m_0}\cup \set{\sigma^{\frown}\seq{k}:k\in Z_{m_0}}$.
  This means that $$\textstyle\bigcup_{\eta\in T_{m_0+1}}\range(\eta)=Z_{m_0}\cup \bigcup_{\eta\in T_{m_0}}\range(\eta).$$
  Finally, the definition of $Z_{m_0}$ implies that
  \[\textstyle
   \bigcup_{i\in A}f_i^{-1}(f_i''\set{u_{m_0}})\subset Z_{m_0}\cup\bigcup_{\eta\in T_{m_0}}\range(\eta)=\bigcup_{\eta\in T_{m_0+1}}\range(\eta),
  \]
  as required.
 \end{proof}

 \begin{claim}\label{closedtree}
  The set $F$ is ($A,\mathcal F$)-closed, and contains $n$ as an element.
 \end{claim}
 
 \begin{proof}
  Since $\seq{n}\in T_0$, $n\in F$.
  To see that $F$ is ($A,\mathcal F$)-closed, take any $j\in A$, and any $w\in f_j^{-1}\br{f_j''F}$.
  We have to show that $w\in F$.
  Since $w\in f_j^{-1}\br{f_j''F}$, there is $m\in F$ such that $f_j(w)=f_j(m)$.
  Since $m\in F=\bigcup_{\eta\in T}\range(\eta)$, there is $\sigma$ in $T$ such that $\sigma(k)=m$ for some $k<\o$.
  Consider $\sigma\rest(k+1)$. 
  Since $\sigma\rest(k+1)\in T$, by Claim \ref{subclaimtree} there is $m_0\ge 1$ such that $\Phi(m_0)=\sigma\rest(k+1)$, that $\sigma\rest(k+1)$ is a leaf node of $T_{m_0}$ and that (note that $u_{m_0}=\sigma(k)=m$)
  \[\textstyle
   \bigcup_{i\in A}f_i^{-1}(f_i''\set{m})\subset \bigcup_{\eta\in T_{m_0+1}}\range(\eta)\subset \bigcup_{\eta\in T}\range(\eta)=F.
  \]
  So $w\in f_j^{-1}(f_j''\set{m})\subset F$ as required.
 \end{proof}
 
 \begin{claim}\label{finitetree}
  The tree $T$ is finite.
 \end{claim}
 
 \begin{proof}
  First we prove that each level of $T$ is finite.
  For $k<\o$ let $T_{(k)}$ be the $k$-th level of $T$, i.e. $T_{(k)}=\set{\sigma\in T: \abs{\sigma}=k}$.
  Clearly $T_{(0)}$ and $T_{(1)}$ are finite.
  So suppose that $T_{(k)}$ is finite.
  Let $T_{(k)}=\set{\sigma_0,\sigma_2,\dots,\sigma_t}$ be enumeration of that level.
  For $s\le t$ let $m_s$ be such that $\Phi(m_s)=\sigma_s$ and that $\sigma_s$ is a leaf node of $T_{m_s}$.
  Note that by the construction of the tree $T$ all nodes at the level $T_{(k+1)}$ are of the form $\sigma_s^{\frown}\seq{r}$ where $s\le t$ and $r\in Z_{m_s}$.
  Since the set $A$ is finite and all functions $f_i$ (for $i\in A$) are finite-to-one, $Z_{m_s}$ is finite for every $s\le t$.
  Thus there are only finitely many nodes of the form $\sigma_s^{\frown}\seq{r}$ where $s\le t$ and $r\in Z_{m_s}$, hence the level $T_{(k+1)}$ must also be finite.
  This proves by induction that each level of $T$ is finite.
  
  Suppose now that $T$ is infinite.
  By K\"{o}nig's lemma, since each level of $T$ is finite, $T$ has an infinite branch $b$.
  By definition of the sets $Z_m$ ($m<\o$), each node of $T$ is 1-1 function, so $b$ is also an injection from $\o$ into $\o$.
  In particular, the range of $b$ is infinite.
  Let $k=\min\set{m<\o:b(m)\ge l''}$, and let $\sigma=b\rest(k+1)$.
  Clearly, $\sigma\in T$.
  By Claim \ref{subclaimtree}, there is $m_0<\o$ such that $\sigma$ is a leaf node of $T_{m_0}$, that $\Phi(m_0)=\sigma$, and that $\bigcup_{i\in A}f_i^{-1}(f_i''\set{\sigma(k)})\subset \bigcup_{\eta\in T_{m_0+1}}\range(\eta)$.
  Since $\sigma(k)=b(k)\ge l''$, Claim \ref{closedset} implies that the set $Y=f_{i_0}^{-1}(f_{i_0}''\set{\sigma(k)})$ is $(A,\mathcal F)$-closed.
  By the construction $T_{m_0+1}=T_{m_0}\cup\set{\sigma^{\frown}\seq{m}:m\in Z_{m_0}}$.
  Since $b$ is an infinite branch, there is $m'\in Z_{m_0}$ such that $b(k+1)=m'$.
  Now $m'\in Z_{m_0}\subset \bigcup_{i\in A}f_i^{-1}\br{f_i''\set{\sigma(k)}}$, the fact that $\sigma(k)\in Y$, and the fact that $Y$ is ($A,\mathcal F$)-closed, together imply that
  $$\textstyle m'\in \bigcup_{i\in A}f_i^{-1}\br{f_i''\set{\sigma(k)}}\subset \bigcup_{i\in A}f_i^{-1}\br{f_i''Y}\subset Y.$$
  
  Consider the node $\tau=\sigma^{\frown}\seq{m'}=b\rest(k+2)$.
  Since $b$ is an infinite branch, it must be that $\tau^{\frown}\seq{b(k+2)}\in T$.
  By Claim \ref{subclaimtree}, there is $m_1$ such that $\tau$ is a leaf node of $T_{m_1}$ and that $\Phi(m_1)=\tau$.
  Clearly, $m_1>m_0$ and $\tau^{\frown}\seq{b(k+2)}\in T_{m_1+1}$.
  Recall that we have already shown that $\bigcup_{i\in A}f_i^{-1}(f_i''\set{\sigma(k)})\subset \bigcup_{\eta\in T_{m_0+1}}\range(\eta)$.
  Thus $Y\subset\bigcup_{\eta\in T_{m_0+1}}\range(\eta)$.
  This, together with the fact that $\tau(k+1)=m'\in Y$, that $Y$ is $(A,\mathcal F)$-closed, and $m_1>m_0$ jointly imply that
  \[\textstyle
   \bigcup_{i\in A}f_i^{-1}(f_i''\set{\tau(k+1)})\subset Y\subset\bigcup_{\eta\in T_{m_0}+1}\range(\eta)\subset \bigcup_{\eta\in T_{m_1}}\range(\eta).
  \]
  This means that
  \[\textstyle
   b(k+2)\in Z_{n_1}=\bigcup_{i\in A}f_i^{-1}(f_i''\set{\tau(k+1)})\setminus \bigcup_{\eta\in T_{m_1}}\range(\eta)=\ps,
  \]
  which is clearly impossible.
  Thus, $T$ is not infinite.
 \end{proof}
 To finish the proof, note that by Claim \ref{closedtree} the set $F$ is $(A,\mathcal F)$-closed and contains $n$ as an element, while by Claim \ref{finitetree} the set $F$ is finite. So $F$ satisfies all the requirements of the conclusion of the lemma.
\end{proof}

The following lemma is the main application of Martin's Axiom.
Again, it does not directly deal with ultrafilters, but with collections of functions.
\begin{lemma}[MA$_{\a}$]\label{mainlemma}
 Let $\mathcal F=\set{f_i:i<\a}\subset \o^{\o}$ be a family of finite-to-one functions. 
 Suppose that for each non-empty finite set $A\subset\a$, and each $n<\o$, there is a finite ($A,\mathcal F$)-closed set $F$ containing $n$ as an element.
 Then there is a finite-to-one function $h\in\o^{\o}$, and a collection $\set{e_i:i<\a}\subset \o^{\o}$ such that for each $i<\a$, there is $l<\o$ such that $h(n)=e_i(f_i(n))$ whenever $n\ge l$.
\end{lemma}

\begin{proof}
 We will apply MA$_{\a}$, so we first define the poset we will be using. Let $\mathcal P$ be the set of all $p=\seq{g_p,h_p}$ such that
 \begin{enumerate}[label=(\Roman*)]
  \item\label{uslov1} $h_p:N_p\to \o$ where $N_p$ is a finite subset of $\o$,
  \item\label{uslov2} $g_p=\seq{g^i_p:i\in A_p}$ where $A_p\in [\a]^{<\aleph_0}$, and $g^i_p:f_i''N_p\to \o$ for each $i\in A_p$,
  \item\label{uslov3} $N_p$ is ($A_p,\mathcal F$)-closed.
 \end{enumerate}
 Define the ordering relation $\le$ on $\mathcal P$ as follows: $q\le p$ iff
 \begin{enumerate}[resume,label=(\Roman*)]
  \item\label{ext1} $N_p\subset N_q$,
  \item\label{ext2} $A_p\subset A_q$,
  \item\label{ext3} $h_q\rest N_p=h_p$,
  \item\label{ext4} $g_q^i\rest f_i''N_p=g_p^i$ for each $i\in A_p$,
  \item\label{ext5} $h_q(n)>h_q(m)$ whenever $m\in N_p$ and $n\in N_q\setminus N_p$,
  \item\label{ext6} $h_q(n)=g_q^i(f_i(n))$ for each $n\in N_q\setminus N_p$ and $i\in A_p$.  
 \end{enumerate}
 It is clear that $\seq{\mathcal P,\le}$ is a partially ordered set.
 
 \begin{claim}\label{addingn}
  Let $p\in \mathcal P$, $n_0<\o$, and suppose that $A\subset \a$ is finite such that $A_p\subset A$.
  Then there is $q\le p$ such that $n_0\subset N_q$ and that $N_q$ is ($A,\mathcal F$)-closed.
 \end{claim}
 
 \begin{proof}
  Applying the assumption of the lemma to the finite set $A$, and each $k\in n_0\cup N_p$, we obtain sets $F_k$ ($k\in n_0\cup N_p$) such that $k\in F_k$ and $f_i^{-1}(f_i''F_k)\subset F_k$ for each $k\in n_0\cup N_p$ and $i\in A$.
  Let $N_q=\bigcup_{k\in n_0\cup N_p}F_k$, let $A_q=A_p$, and let $t=\max\set{h_p(k)+1:k\in N_p}$.
  Finally, define
  \[
   h_q(n)=\left\{\begin{array}{l} h_p(n),\ \mbox{if}\ n\in N_p\\
                                       t,\ \mbox{if}\ n\in N_q\setminus N_p\end{array}\right.\ \mbox{and}\ g^i_q(k)=\left\{\begin{array}{l} g^i_p(k),\ \mbox{if}\ k\in f_i''N_p\\
                                       t,\ \mbox{if}\ k\in f_i''N_q\setminus f_i''N_p\end{array}\right.\ \mbox{for}\ i\in A_q.
  \]
  Let $g_q$ denote the sequence $\seq{g^i_q:i\in A_q}$.
  Clearly $n_0\subset N_q$.
  By Lemma \ref{finiteunion}, $N_q$ is ($A,\mathcal F$)-closed.
  We still have to show that $q=\seq{g_q,h_q}\in \mathcal P$ and $q\le p$.
  Since $h_q$ is defined on $N_q$ and $N_q$ finite, since $A_q=A_p$ and $g_p^i$ is defined on $f_i''N_q$ for $i\in A_p$, and since $N_q$ is $(A_q,\mathcal F)$-closed, conditions \ref{uslov1}-\ref{uslov3} are satisfied by $q$.
  Thus $q\le p$.
  Next we show $q\le p$.
  Conditions \ref{ext1}-\ref{ext4} are obviously  satisfied by the definition of $q$.
  Since $h_q(n)=t>h_p(k)=h_q(k)$ for each $n\in N_q\setminus N_p$ and $k\in N_p$, conditions \ref{ext4} is also satisfied.
  So we still have to check \ref{ext6}.
  Take any $i\in A_p$ and $n\in N_q\setminus N_p$.
  By the definition of $h_q$, we have $h_q(n)=t$.
  Once we prove that $f_i(n)\in f_i''N_q\setminus f_i''N_p$, we will be done because in that case the definition of $g^i_p$ implies that $g^i_p(f_i(n))=t$ as required.
  So suppose the contrary, that $f_i(n)\in f_i''N_p$.
  Since $p$ is a condition and $A_q=A_p$, it must be that $n\in f_i^{-1}(f_i''N_p)\subset N_p$.
  But this contradicts the choice of $n$.
  Thus condition \ref{ext6} is also satisfied and $q\le p$.
 \end{proof}
 
 \begin{claim}\label{addingj}
  Let $p\in\mathcal P$, and $j_0<\a$. Then there is $q\le p$ such that $j_0\in A_q$.
 \end{claim}
 
 \begin{proof}
  Let $A_q=A_p\cup\set{j_0}$.
  Applying Claim \ref{addingn} to $A_q$ and $n=0$, we obtain a condition $p'\le p$ such that $N_{p'}$ is $(A_q,\mathcal F)$-closed.
  Let $N_q=N_{p'}$, $h_{q}=h_{p'}$, and $g^i_q=g^i_{p'}$ for $i\in A_p$.
  Define $g^{j_0}_q(k)=0$ for each $k\in f_{j_0}''N_{p'}$, and let $g_p$ denote the sequence $\seq{g^i_q:i\in A_q}$.
  Since $j_0\in A_q$, to finish the proof of the claim it is enough to show that $q=\seq{g_q,h_q}\in \mathcal P$, and that $q\le p'$.
  Conditions \ref{uslov1}-\ref{uslov3} are clear from the definition of $q$ because $p'\in\mathcal P$ and $N_q=N_{p'}$, $h_q=h_{p'}$, $g_q^{j_0}:f_{j_0}''N_q\to\o$, and $g_q^i=g_p^i$ for $i\in A_q\setminus\set{j_0}$.
  Conditions \ref{ext1}-\ref{ext4} are clear by the definition of $h_q$ and $g_q$.
  Conditions \ref{ext5} and \ref{ext6} are vacuously true because $N_{p'}=N_q$.
  Thus the claim is proved.
 \end{proof}
 
 \begin{claim}\label{merging}
  If $p,q\in \mathcal P$ are such that $h_p=h_q$ and $g^i_p=g^i_q$ for $i\in A_p\cap A_q$, then $p$ and $q$ are compatible in $\mathcal P$.
 \end{claim}
 
 \begin{proof}
  We proceed to define $r\le p,q$.
  Let $N=N_p=N_q$.
  Let $$t=\max\set{h_p(n)+1:n\in N}.$$
  Applying the assumption of the lemma to $A=A_p\cup A_q$, and each $k\in N$, we obtain ($A,\mathcal F$)-closed sets $F_k$ ($k\in N$).
  By Claim \ref{finiteunion}, the set $N_r=\bigcup_{k\in N}F_k$ is ($A,\mathcal F$)-closed.
  Let $A_r=A$, and define
  \[
   h_r(n)=\left\{\begin{array}{l} h_p(n),\mbox{ if}\ n\in N\\[1mm]
                                       t,\mbox{ if}\ n\in N_r\setminus N\end{array}\right.\hskip-1.5mm\mbox{and }g^i_r(k)=\left\{\begin{array}{l} g^i_p(k),\mbox{ if}\ i\in A_p\mbox{ and }k\in f_i''N\\[1mm]
                                       g^i_q(k),\mbox{ if}\ i\in A_q\mbox{ and }k\in f_i''N\\[1mm]
                                       t,\ \mbox{ if}\ k\in f_i''N_r\setminus f_i''N\end{array}\right.\hskip-1.5mm,
  \]
  for $i\in A_r$. Let $g_r$ denote the sequence $\seq{g^i_r:i\in A_r}$.
  As we have already mentioned, $r=\seq{h_r,g_r}$ satisfies \ref{uslov3}, and it clear that \ref{uslov1} and \ref{uslov2} are also true for $r$.
  To see that $r\le p$ and $r\le q$ note that conditions \ref{ext1}-\ref{ext5} are clearly satisfied.
  We will check that $r$ and $p$ satisfy \ref{ext6} also.
  Take any $n\in N_r\setminus N$ and $i\in A_p$.
  By the definition of $h_r$, $h_r(n)=t$.
  By the definition of $g_r^i$, if $f_i(n)\in f_i''N_r\setminus f_i''N$, then $g^i_r(f_i(n))=t=h_r(n)$.
  So suppose this is not the case, i.e. that $f_i(n)\in f_i''N$.
  This would mean that $n\in f_i^{-1}(f_i''N)$, which is impossible because $n\notin N$ and $N$ is $(A,\mathcal F)$-closed because $p\in \mathcal P$.
  Thus we proved $r\le p$.
  In the same way it can be shown that $r\le q$.
 \end{proof}
 
 \begin{claim}\label{ccc}
  The poset $\mathcal P$ satisfies the countable chain condition.
 \end{claim}
 
 \begin{proof}
  Let $\seq{p_{\xi}:\xi<\o_1}$ be an uncountable set of conditions in $\mathcal P$.
  Since $h_{p_{\xi}}\in [\o\times\o]^{<\o}$ for each $\xi<\o_1$, there is an uncountable set $\Gamma\subset\o_1$, and $h\in [\o\times\o]^{<\o}$ such that $h_{p_{\xi}}=h$ for each $\xi\in \Gamma$.
  Consider the set $\seq{A_{p_{\xi}}:\xi\in\Gamma}$.
  By the $\Delta$-system lemma, there is an uncountable set $\Delta\subset\Gamma$, and a finite set $A\subset\a$ such that $A_{p_{\xi}}\cap A_{p_{\eta}}=A$ for each $\xi,\eta\in\Delta$.
  Since $\Delta$ is uncountable and $g_{p_{\xi}}^i\in [\o\times\o]^{<\o}$ for each $i\in A$ and $\xi\in\Delta$, there is an uncountable set $\Theta\subset\Delta$ and $g_i$ for each $i\in A$, such that $g^i_{p_{\xi}}=g_i$ for each $\xi\in\Theta$ and $i\in A$.
  Let $\xi$ and $\eta$ in $\Theta$ be arbitrary.
  By Claim \ref{merging}, $p_{\xi}$ and $p_{\eta}$ are compatible in $\mathcal P$.
 \end{proof}
 
 Consider sets $D_{j}=\set{p\in \mathcal P: j\in A_p}$ for $j\in\a$, and $D_{m}=\set{p\in \mathcal P:m\in N_p}$ for $m\in\o$.
 By Claim \ref{addingj} and Claim \ref{addingn}, these sets are dense in $\mathcal P$.
 By MA$_{\a}$ there is a filter $\mathcal G\subset \mathcal P$ intersecting all these sets.
 Clearly, $h=\bigcup_{p\in \mathcal G}h_p$ and $e_i=\bigcup_{p\in\mathcal G}g_p^i$, for each $i\in \a$, are functions from $\o$ into $\o$.
 We will prove that these functions satisfy the conclusion of the lemma.
 First we will prove that $h$ is finite-to-one.
 Take any $m\in \o$ and let $k=h(m)$.
 By the definition of $h$, there is $p\in \mathcal G$ such that $h_p(m)=k$.
 Suppose that $h^{-1}(\set{k})\not\subset N_p$.
 This means that there is an integer $m_0\notin N_p$ such that $h(m_0)=k$.
 Let $q\in \mathcal G$ be such that $h_q(m_0)=k$.
 Now for a common extension $r\in\mathcal G$ of both $p$ and $q$, it must be that $h_r(m_0)=h_p(m)$, contradicting the fact that $r\le p$, in particular condition \ref{ext5} is violated in this case.
 We still have to show that for each $i\in \a$, there is $l\in\o$ such that $h(n)=e_i(f_i(n))$ whenever $n\ge l$.
 So take $i\in\a$. 
 By Claim \ref{addingj}, there is $p\in\mathcal G$ such that $i\in A_p$.
 Let $l=\max(N_p)+1$.
 We will prove that $l$ is as required.
 Take any $n\ge l$.
 By Claim \ref{addingn}, there is $q\in\mathcal G$ such that $n\in q$.
 Let $r\in \mathcal G$ be a common extension of $p$ and $q$.
 Since $n\notin N_p$ and $r\le p$, it must be that $h_r(n)=g_r^i(f_i(n))$, according to condition \ref{ext6}.
 Hence $h(n)=e_i(f_i(n))$, as required.
\end{proof}

Before we move to the next lemma let us recall that if $c$ is any element of the model $M$, then $\cu=\set{X\subset \o: c\in {}^*X}$ is ultrafilter on $\o$.

\begin{lemma}\label{ensuringensuring}
	Let $\a<\c$ be an ordinal.
	Let $\seq{M_i:i<\a}$ be a $\subset$-decreasing sequence of principal submodels of $M$, i.e. each $M_i$ is generated by a single element $a_i$ and $M_j\subset M_i$ whenever $i<j<\a$.
	Let each $M_i$ ($i<\a$) be cofinal with $M_0$.
	Suppose that $\mathcal U_0=\set{X\subset \o: a_0\in \prescript{*}{}X}$ is a $P_{\c}$-point.
	Then there is a family  $\set{f_i:i<\a}\subset \o^{\o}$ of finite-to-one functions such that $\prescript{*}{}f_i(a_0)=a_i$ for $i<\a$, and that for $i,j<\a$ with $i<j$, there is $l<\o$ such that $f_j(n)=f_j(m)$ whenever $f_i(n)=f_i(m)$ and $m,n\ge l$.
\end{lemma}

\begin{proof}
	Let $i<j<\a$. Since $M_j\subset M_i$, and $M_i$ is generated by $a_i$, there is a function $\varphi_{ij}:\o\to\o$ such that $\prescript{*}{}\varphi_{ij}(a_i)=a_j$.
	Since $M_j$ is cofinal with $M_i$, by Lemma in \cite[page 104]{blassmodeltheory}, if $i<j<\a$, then there is a set $Y_{ij}\subset\o$ such that $a_i\in \prescript{*}{}Y_{ij}$ and that $\varphi_{ij}\rest Y_{ij}$ is finite-to-one.
	For $i<\a$ let $g_i:\o\to\o$ be defined as follows: if $n<\o$, then for $n\notin Y_{0i}$ let $g_i(n)=n$, while for $n\in Y_{0i}$ let $g_i(n)=\varphi_{0i}(n)$.
	Note that $\prescript{*}{}g_i(a_0)=a_i$, and that $g_i$ is finite-to-one.
	The latter fact follows since $g_i$ is one-to-one on $\o\setminus Y_{0i}$ and on $Y_{0i}$ it is equal to $\varphi_{0i}$, which is finite-to-one on $Y_{0i}$.
	Now by the second part of Lemma on page 104 in \cite{blassmodeltheory}, for $i<j<\a$ there is a finite-to-one function $\pi_{ij}:\o\to\o$ such that $\prescript{*}{}\varphi_{ij}(a_i)=\prescript{*}{}\pi_{ij}(a_i)$.
	Note that this means that $\prescript{*}{}g_j(a_0)=\prescript{*}{}\pi_{ij}(\prescript{*}{}g_i(a_0))$ for $i<j<\a$, i.e. the set $X_{ij}=\set{n\in\o:g_j(n)=\pi_{ij}(g_i(n))}$ is in $\cu_0$.
	Since $\a<\mathfrak c$ and $\cu_0$ is a $P_{\c}$-point, there is a set $X\subset \o$ such that $X\in\cu_0$ and that the set $X\setminus X_{ij}$ is finite whenever $i<j<\a$.
	Consider the sets $W_i=g_i''X$ for $i<\a$.
	For each $i<\a$, let $W_i=W_i^0\cup W_i^1$ where $W_i^0\cap W_i^1=\ps$ and both $W_i^0$ and $W_i^1$ are infinite.
	Fix $i<\a$ for a moment.
	We know that
	\[\textstyle
	X=\left(X\cap \bigcup_{n\in W_i^0}g_i^{-1}(\set{n})\right)\cup \left(X\cap \bigcup_{n\in W_i^1}g_i^{-1}(\set{n})\right).
	\]
	Since $X\in\cu_0$ and $\cu_0$ is an ultrafilter, we have that either $X\cap \bigcup_{n\in W_i^0}g_i^{-1}(\set{n})\in\cu_0$ or $X\cap \bigcup_{n\in W_i^1}g_i^{-1}(\set{n})\in\cu_0$.
	Suppose that $Y=X\cap \bigcup_{n\in W_i^0}g_i^{-1}(\set{n})\in \cu_0$ (the other case would be handled similarly).
	Note that by the definition of $\cu_0$ we know that $a_0\in {}^*Y$.
	Define $f_i:\o\to\o$ as follows: for $n\in Y$ let $f_i(n)=g_i(n)$, while for $n\notin Y$ let $f_i(n)=W_i^1(n)$.
	Now that functions $f_i$ are defined, unfix $i$.
	We will prove that $\mathcal F=\set{f_i:i<\a}$ has all the properties from the conclusion of the lemma.
	Since each $g_i$ ($i<\a$) is finite-to-one, it is clear that $f_i$ is also finite-to-one.
	Again, this is because $g_i$ is finite-to-one on $\o$, and outside of $Y$ the function $f_i$ is defined so that it is one-to-one.
	Since $\prescript{*}{}g_i(a_0)=a_i$ and $a_0\in {}^*Y$, it must be that $\prescript{*}{}f_i(a_0)=a_i$ for each $i<\a$.
	Now we prove the last property.
	Suppose that $i<j<\a$.
	Since the set $X\setminus X_{ij}$ is finite and $Y\subset X$, there is $l<\o$ so that $Y\setminus l\subset X_{ij}$.
	Take $m,n\ge l$, and suppose that $f_i(n)=f_i(m)$.
	There are three cases.
	First, $n,m\notin Y$.
	In this case, $f_i(n)=f_i(m)$ implies that $W_i^1(n)=W_i^1(m)$, i.e. $n=m$.
	Hence $f_j(n)=f_j(m)$.
	Second, $m\in Y$ and $n\notin Y$.
	Since $m\in Y$, $g_i(m)=f_i(m)=f_i(n)$ so $f_i(n)\in W_i^0$.
	On the other hand, since $n\notin Y$, $f_i(n)=W_i^1(n)$.
	Thus we have $f_i(n)\in W_i^0\cap W_i^1$ which is in contradiction with the fact that $W_i^0\cap W_i^1=\ps$.
	So this case is not possible.
	Third, $m,n\in Y$.
	In this case $f_i(n)=f_i(m)$ implies that $g_i(n)=g_i(m)$.
	Since $m,n\in Y\setminus l\subset X_{ij}$ it must be that $f_j(n)=g_j(n)=\pi_{ij}(g_i(n))=\pi_{ij}(g_i(m))=g_j(m)=f_j(m)$ as required.
	Thus the lemma is proved.
\end{proof}

\begin{lemma}[MA$_{\a}$]\label{ensuringtheorem3}
 Let $\seq{M_i:i<\a}$ be a $\subset$-decreasing sequence of principal, and pairwise cofinal submodels of $M$.
 Suppose that $\mathcal U_0=\set{X\subset \o: a_0\in \prescript{*}{}X}$ is a $P_{\c}$-point, where $a_0$ generates $M_0$.
 Then there is an element $c\in\bigcap_{i<\a}M_i$ which generates a principal model cofinal with all $M_i$ ($i<\a$). 
\end{lemma}

\begin{proof}
 Let $a_i$ for $i<\a$ be an element generating $M_i$.
 By Lemma \ref{ensuringensuring} there is a family $\mathcal F=\set{f_i:i<\a}\subset \o^{\o}$ of finite-to-one functions such that $\prescript{*}{}f_i(a_0)=a_i$ for $i<\a$, and that for $i,j<\a$ with $i<j$, there is $l<\o$ such that $f_j(n)=f_j(m)$ whenever $f_i(n)=f_i(m)$ and $m,n\ge l$.
 By Lemma \ref{ensuringcondition}, for each finite $A\subset \a$, and each $n<\o$, there is a finite ($A,\mathcal F$)-closed set containing $n$ as an element.
 Now using MA$_{\a}$, Lemma \ref{mainlemma} implies that there is a finite-to-one function $h\in\o^{\o}$, and a collection $\set{e_i:i<\a}\subset\o^{\o}$ such that for each $i<\a$ there is $l<\o$ such that $h(n)=e_i(f_i(n))$ whenever $n\ge l$.
 
 Let $c=\prescript{*}{}h(a_0)$, and let $M_{\a}$ be a model generated by $c$.
 By Lemma in \cite[pp. 104]{blassmodeltheory}, $M_{\a}$ is cofinal with $M_0$.
 Thus $M_{\a}$ is a principal model cofinal with all $M_i$ ($i<\a$).
 To finish the proof we still have to show that $c\in \bigcap_{i<\a}M_i$.
 Fix $i<\a$.
 Let $l<\o$ be such that $h(n)=e_i(f_i(n))$ for $n\ge l$.
 If $a_0<l$, then $M_j=M$ for each $j<\a$ so the conclusion is trivially satisfied. So $a_0\prescript{*}{}\ge l$.
 Since the sentence $(\forall n) [n\ge l\sledi h(n)=e_i(f_i(n))]$ is true in $M$, it is also true in $M_0$.
 Thus $c=\prescript{*}{}h(a_0)=\prescript{*}{}e_i(\prescript{*}{}f_i(a_0))=\prescript{*}{}e_i(a_i)\in M_i$ as required.
\end{proof}

\begin{theorem}[MA$_{\a}$]\label{theorem3}
 %Let $\a<\k$ be an ordinal.
 Let $M_i$ ($i<\a$) be a collection of pairwise cofinal submodels of $M$.
 Suppose that $M_{0}$ is principal, and that $\mathcal U_0=\set{X\subset \o: a_0\in \prescript{*}{}X}$ is a $P_{\c}$-point, where $a_0$ generates $M_0$.
 Then $\bigcap_{i<\a}M_i$ contains a principal submodel cofinal with each $M_i$.
\end{theorem}

\begin{proof}
 We define models $M'_i$ for $i\le\a$ as follows. $M'_0=M_0$. If $M'_i$ is defined, then $M'_{i+1}$ is a principal submodel of $M'_i\cap M_{i+1}$ cofinal with $M'_i$ and $M_{i+1}$. This model exists by Corollary in \cite[pp. 105]{blassmodeltheory}. If $i\le \a$ is limit, then the model $M'_i$ is a principal model cofinal with all $M'_j$ ($j<i$). This model exists by Lemma \ref{ensuringtheorem3}. Now the model $M'_{\a}$ is as required in the conclusion of the lemma.
\end{proof}

\begin{theorem}[MA$_{\a}$]\label{maintheorem}
 %Let $\a<\k$.
 Suppose that $\set{\cu_i:i<\a}$ is a collection of P-points.
 Suppose moreover that $\cu_0$ is a P$_{\c}$-point such that $\cu_i\le_{RK}\cu_0$ for each $i<\a$.
 Then there is a P-point $\cu$ such that $\cu\le_{RK} \cu_i$ for each $i$.
\end{theorem}

\begin{proof}
 By Theorem 3 of \cite{blass}, $\o^{\o}/\cu_i$ is isomorphic to an elementary submodel $M_i$ of $\o^{\o}/\cu_0$.
 Since all $\cu_i$ ($i<\a$) are non-principal, each model $M_i$ ($i<\a$) is non-standard.
 By Corollary in \cite[pp. 150]{blass}, each $M_i$ ($i<\a$) is cofinal with $M_0$.
 This implies that all the models $M_i$ ($i<\a$) are pairwise cofinal with each other.
 By Theorem \ref{theorem3} there is a principal model $M'$ which is a subset of each $M_i$ ($i<\a$) and is cofinal with $M_0$.
 Since $M'$ is principal, there is an element $a$ generating $M'$.
 Let $\cu=\set{X\subset \o:a\in \prescript{*}{}X}$.
 Then $\o^{\o}/\cu\cong M'$.
 Since $M'$ is cofinal with $M_0$, $M'$ is not the standard model.
 Thus $\cu$ is non-principal.
 Now $M'\prec M_i$ ($i<\a$) implies that $\cu\le_{RK}\cu_i$ (again using Theorem 3 of \cite{blass}).
 Since $\cu$ is Rudin-Keisler below a $P$-point, $\cu$ is also a $P$-point.
\end{proof}
\begin{corollary}[MA] \label{cor:lower}
 If a collection of fewer than $\c$ many ${P}_{\c}$-points has an upper bound which is a ${P}_{\c}$-point, then it has a lower bound.
\end{corollary}
\begin{corollary}[MA] \label{cor:closed}
 The class of ${P}_{\c}$-points is downwards $< \c$-closed under ${\leq}_{RK}$.
 In other words, if $\alpha < \c$ and $\langle {\cu}_{i}: i < \alpha \rangle$ is a sequence of ${P}_{\c}$-points such that $\forall i < j < \alpha\left[{\cu}_{j} \: {\leq}_{RK} \: {\cu}_{i}\right]$, then there is a ${P}_{\c}$-point $\cu$ such that $\forall i < \alpha\left[ \cu \: {\leq}_{RK} \: {\cu}_{i} \right]$.
\end{corollary}

\end{document}